\documentclass[12pt,reqno]{amsart}

\usepackage{hyperref}

\newtheorem{Theorem}{Theorem}[section]
\newtheorem{Proposition}[Theorem]{Proposition}
\newtheorem{Corollary}[Theorem]{Corollary}
\newtheorem{manualCorollaryinner}{Corollary}
\newenvironment{manualCorollary}[1]{%
  \renewcommand\themanualCorollaryinner{#1}%
  \manualCorollaryinner
}{\endmanualCorollaryinner}
\newtheorem{Lemma}[Theorem]{Lemma}

\theoremstyle{remark}
\newtheorem{Remark}[Theorem]{Remark}

\numberwithin{equation}{section}

\allowdisplaybreaks

\newcommand{\myref}[1]{(\ref{#1})}
\newcommand{\kL}{k\in\Lambda^{(r-1)}}

\author[Michael J.\ Schlosser]{Michael J.\ Schlosser$^{*}$}
\address{Fakult\"at f\"ur Mathematik, Universit\"at Wien,
Oskar-Morgenstern-Platz~1, A-1090 Vienna, Austria}
\email{michael.schlosser@univie.ac.at}
\thanks{$^*$Research partly supported by FWF Austrian Science Fund
grants F50-08 and P32305.}

\title[Bilateral Rogers--Ramanujan type identities]{Bilateral identities
of the Rogers--Ramanujan type}


\dedicatory{Dedicated to the memory of Srinivasa Ramanujan}

\begin{document}

\begin{abstract}
We derive by analytic means a number of bilateral identities of the
Rogers--Ramanujan type.
Our results include bilateral extensions of the Rogers--Ramanujan
and the G\"ollnitz--Gordon identities, and of related identities
by Ramanujan, Jackson, and Slater.
We give corresponding results for multisums including multilateral
extensions of the Andrews--Gordon identities, of the Andrews--Bressoud
generalization of the G\"ollnitz--Gordon identities, of
Bressoud's even modulus identities, and other identities.
Our closed form bilateral and multilateral
summations appear to be the very first of their kind.
%
\end{abstract}

\maketitle


\section{Introduction}
For complex variables $a$ and $q$ with $|q|<1$ and
$k\in\mathbb Z\cup\{\infty\}$,
the $q$-shifted factorials are defined as follows (cf.\ \cite{GR}):
\begin{equation*}
  (a;q)_\infty:=\prod_{j=1}^\infty(1-aq^{j-1}),\qquad\text{and}
  \qquad (a;q)_k=\frac{(a;q)_\infty}{(aq^k;q)_\infty}.
\end{equation*}
Specifically,
\begin{equation*}
(a;q)_k:=
\begin{cases}1&\text{for $k=0$},\\*
\prod_{j=1}^k(1-aq^{j-1})
&\text{for $k>0$},\\*
\prod_{j=1}^{-k}(1-aq^{-j})^{-1}
&\text{for $k<0$}.
\end{cases}
\end{equation*}
The variable $q$ is referred to as the \textit{base}.
For brevity, we frequently use the compact notation
\begin{equation*}
(a_1,\dots,a_m;q)_k=(a_1;q)_k\cdots(a_m;q)_k,
\end{equation*}
where $m$ is a positive integer.
Unless stated otherwise, all the summations in this paper converge
absolutely (subject to the condition $|q|<1$, which we assume).

The first and second Rogers--Ramanujan identities,
\begin{subequations}\label{rr}
\begin{align}\label{rra}
\sum_{k=0}^\infty\frac{q^{k^2}}{(q;q)_k}&=
\frac1{(q,q^4;q^5)_\infty},\\*
\label{rrb}
\sum_{k=0}^\infty\frac{q^{k(k+1)}}{(q;q)_k}&=
\frac1{(q^2,q^3;q^5)_\infty},
\end{align}
\end{subequations}
have a prominent history. They were first discovered and proved in 1894
by Rogers~\cite{Ro2}, and then independently rediscovered by the legendary
Indian mathematician Ramanujan some time before 1913
(cf.\ Hardy~\cite{Ha}).
They were also independently discovered and proved in 1917 by Schur~\cite{Sc}.
About the pair of identities in \myref{rr} Hardy~\cite[p.~xxxiv]{HSW} remarked
\begin{quotation}
`It would be difficult to find more beautiful formulae than the
``Rogers--Ramanujan'' identities, $\ldots$'
\end{quotation}
It is not clear how Ramanujan was led to discover \myref{rr}.
Bhatnagar~\cite{Bh} describes a method to conjecture these identities.
A basic hypergeometric proof of \myref{rr} was found by Watson~\cite{Wa},
who observed that these identities can be obtained from
the (now called) Watson transformation by taking suitable limits
and applying instances of Jacobi's triple product identity.
Watson's proof is not the only early proof using hypergeometric machinery.
In \cite{Ro2}, Rogers obtained directly an identity which nowadays
is called the ``Rogers--Selberg identity'' (and which happens to be
a special case of the Watson transformation that was discovered much
later) from which
the two Rogers--Ramanujan identities follow by specialization
and instances of Jacobi's triple product identity.

The Rogers--Ramanujan identities are deep identities which
have found interpretations in combinatorics, number theory,
orthogonal polynomials, probability theory, statistical mechanics,
representations of Lie algebras, vertex operator algebras, knot theory,
and conformal field
theory~\cite{AnRR,An86,AD,B,BM,Br2,BZ,BP,Co,Fu,GIS,GM,Go0,LM,LW,Mac,Pr,Sc,W}.
(We do not claim that the list of areas given is complete. Moreover,
the selected references are only representative samples of papers
on interpretations of the Rogers--Ramanujan identities.)
A recent highlight in the theory concerns the construction
of these identities for higher-rank Lie algebras~\cite{GOW}.

A pair of identities similar to \myref{rr} are
the first and second G\"ollnitz--Gordon identities,
\begin{subequations}\label{gg}
\begin{align}\label{gga}
\sum_{k=0}^\infty\frac{(-q;q^2)_k}{(q^2;q^2)_k}
q^{k^2}&=
\frac 1{(q,q^4,q^7;q^8)_\infty},\\*
\label{ggb}
\sum_{k=0}^\infty\frac{(-q;q^2)_k}{(q^2;q^2)_k}
q^{k(k+2)}&=
\frac 1{(q^3,q^4,q^5;q^8)_\infty}.
\end{align}
\end{subequations}
These appeared in a combinatorial study of integer partitions
in unpublished work by G\"ollnitz in 1960 (\cite{Goe0}, see also \cite{Goe})
and were rediscovered in 1965 by Gordon~\cite{Go}.
However, they were recorded more than 40 years earlier by
Ramanujan in his lost notebook, see \cite[p.~36--37, Entries 1.7.11--12]{AB},
and were also published in 1952 by Slater~\cite{Sl} as specific entries
in her famous list of 130 identities of the Rogers--Ramanujan type. 
The systematic study of such identities
had been initiated by Bailey~\cite{Ba1,Ba2} a few years earlier.
A more complete list of identities of the Rogers--Ramanujan type
was recently given by McLaughlin, Sills and Zimmer~\cite{McLSiZi}.
Further such identities were given by Chu and Zhang~\cite{CZ}.
McLaughlin, Sills and Zimmer's list is reproduced
(with some typographical errors corrected)
in Appendix~A of Sills' recent book \cite{Si}, which provides an
excellent introduction to the Rogers--Ramanujan identities.

The analytic identities in \myref{rr} and \myref{gg}
admit partition-theoretic interpretations (cf.\ \cite{An}).
Because of the specific form of the $q$-products on the right-hand sides, 
the identities in \myref{rr}, resp.\ \myref{gg},
are often classified as modulus $5$ and modulus $8$ identities, respectively.

Another identity intimately linked to Ramanujan's name is the following
summation formula (cf.\ \cite[Appendix~(II.29)]{GR})
\begin{equation}\label{1psi1}
\sum_{k=-\infty}^\infty\frac{(a;q)_k}{(b;q)_k}z^k=
\frac{(q,az,q/az,b/a;q)_\infty}{(b,z,b/az,q/a;q)_\infty},\qquad|b/a|<|z|<1.
\end{equation}
This identity, commonly known as Ramanujan's ${}_1\psi_1$
summation, is a bilateral extension of
the $q$-binomial theorem (cf.\ \cite[Appendix~(II.3)]{GR})
\begin{equation}\label{1phi0}
\sum_{k=0}^\infty\frac{(a;q)_k}{(q;q)_k}z^k=
\frac{(az;q)_\infty}{(z;q)_\infty},\qquad|z|<1,
\end{equation}
which is a fundamental identity in the theory of basic
hypergeometric series.
Hardy described \myref{1psi1}, which Ramanujan had
noted but did not publish, as ``a remarkable formula with many
parameters''~\cite[Eq.~(12.12.2)]{Ha2}. Importantly,
\myref{1psi1} contains Jacobi's triple product identity \myref{jtpi}
as a limiting case, an identity which plays a key role in many of the proofs
of identities of the Rogers--Ramanujan type (and which we also
make heavy use of in this paper).

Knowing that the $q$-binomial theorem \myref{1phi0}
extends to a bilateral summation, one can ask the same question
about the Rogers--Ramanujan and G\"ollnitz--Gordon identities in \myref{rr}
and \myref{gg}. While some authors have studied properties of
bilateral series which extend the series in \myref{rr} (see \cite{An70,EI,IZ}),
no closed form bilateral summations which include the evaluations in
\myref{rr} (or \myref{gg}) as special cases have been obtained yet.

In this paper, we derive \textit{bilateral} extensions of the
Rogers--Ramanujan and G\"ollnitz--Gordon identities in \myref{rr}
and \myref{gg} and provide a number of related results.
Our main results for single series are given in
Section~\ref{sec:thms}, together with several noteworthy corollaries.
The proofs of the main results of Section~\ref{sec:thms},
namely Theorems~\ref{thm5}, \ref{thm8} and \ref{thm:bilcc}
are deferred to Section~\ref{sec:bilRR}. The proofs are analytic and
involve a method similar to that used by Watson in \cite{Wa}
to prove the classical Rogers--Ramanujan identities.
In particular, we utilize suitable limiting cases of a bilateral
basic hypergeometric transformation formula of Bailey
in combination with special instances of
Jacobi's triple product identity to establish the respective identities.
In Section~\ref{sec:multRR} multisum extensions of our results
are given, which in particular include multilateral extensions
of the Andrews--Gordon identities and of the Andrews--Bressoud
generalization of the G\"ollnitz--Gordon identities,
in addition to other multisum identities.
We end our paper with some concluding remarks in Section~\ref{sec:conc}.


\section{Main results and corollaries in the single series case}
\label{sec:thms}

Our first result is a bilateral extension of the two
Rogers--Ramanujan identities in \myref{rr}.

\begin{Theorem}\label{thm5}
We have the following two bilateral summations:
\begin{subequations}\label{bilrrg}
\begin{align}\label{bilrrg1}
\sum_{k=-\infty}^\infty\frac{q^{k^2}}{(zq;q)_k}z^{2k}
={}&\frac{(1/z;q)_\infty}
{(1/z^2,z^2q;q)_\infty}(q^5;q^5)_\infty\notag\\*
&\times
\left[(z^5q^3,z^{-5}q^2;q^5)_\infty
+z^{-1}(z^5q^2,z^{-5}q^3;q^5)_\infty\right],\\*
\label{bilrrg2}
\sum_{k=-\infty}^\infty\frac{q^{k(k+1)}}{(zq;q)_k}z^{2k}
={}&\frac{(1/z;q)_\infty}
{(1/z^2,z^2q;q)_\infty}\notag(q^5;q^5)_\infty\\*
&\times
\left[(z^5q^4,z^{-5}q;q^5)_\infty+z^{-3}(z^5q,z^{-5}q^4;q^5)_\infty\right],
\end{align}
\end{subequations}
for complex $z$ such that $z\not\in\{q^{-1},q^{-2},\ldots\}$.
\end{Theorem}

The $z\to 1$ limit of \myref{bilrrg1} gives \myref{rra},
while the $z\to 1$ limit of \myref{bilrrg2} gives \myref{rrb}.

\begin{Remark}\label{remga}
As was kindly brought to the author's attention by George Andrews
after being shown an earlier version of this paper, a result related
to the series on the left-hand side of \myref{bilrrg2}
was found by Andrews in 1970~\cite[Thm.~3]{An70}, namely:
\begin{subequations}
\textit{Let}
\begin{align}\label{gaid}
&g(z)=(-z;q)_\infty\sum_{k=-\infty}^\infty\frac{q^{k(k-1)}}{(-z;q)_k}z^{2k},\\*
\intertext{\textit{then}}
&\frac{g(z)+g(-z)}2=\frac{(q^2,-z^2,-z^{-2}q^2;q^2)_\infty}
{(q;q^2)_\infty(q^4,q^{16};q^{20})_\infty}.
\end{align}
\end{subequations}
The last expression shows that the even part
of $g(z)$ (or of $g(-z)$) can be expressed in closed form.
The method used in \cite{An70} can also be used to
express the odd part of $g(z)$ in closed form (which was not
done in \cite{An70}),
which together with \myref{gaid} can be used to
obtain an evaluation for $g(z)$, similar to \myref{bilrrg2}.
\end{Remark}
\begin{Remark}\label{rem:ref}
  While in Section~\ref{sec:bilRR}
we show how one can derive the two identities in \myref{bilrrg}
and the the other results listed in this section by using
a powerful transformation formula for bilateral basic
hypergeometric series derived by Bailey
(see \myref{bil68}) without appealing to
the classical Rogers--Ramanujan identities in \myref{rr},
the anonymous Referee pointed out a way how one can easily prove
(actually, verify) the identities in \myref{bilrrg} by an analytic,
functional equation approach and making use of \myref{rr}.
We sketch the details for the Referee's verification of \myref{bilrrg1}.
The details for verifying \myref{bilrrg2} are similar.

\textit{Verification proof of} \myref{bilrrg1}. \
Both sides of \myref{bilrrg1} satisfy the functional equation
$$
f(z;q)=\frac{z^2q}{1-zq}\, f(zq;q).
$$
This implies that the ratio of the left- and right-hand sides
is an elliptic (multiplicatively $q$-periodic) function.
The right-hand side of \myref{bilrrg1} has poles at
$z=-q^m$, $z=\pm q^{1/2+m}$ and $z=\pm q^{-1/2-m/2}$,
for $m$ a nonnegative integer.

Because of the functional equation it is enough to consider the poles at
$z=-1$, $z=\pm q^{-1/2}$ and $z=\pm q^{-1}$.
The four poles at $z=-1$, $z=q^{-1/2}$, $z=-q^{-1/2}$ and $z=-q^{-1}$
have zero residue and the only pole with non-zero residue is
the pole at $z=q^{-1}$.
Indeed, both sides of the identity have poles at $z=q^{-1-m}$
for nonnegative integers $m$. Focusing on $z=q^{-1}$
(again, by the $q$-periodicity this is enough)
one checks that the residue on the sum side at
$z=q^{-1}$ is
$$-q^{-2}\sum_{k>0} \frac{q^{(k-1)^2}}{(q;q)_{k-1}} =
-q^{-2} \sum_{k\ge 0} \frac{q^{k^2}}{(q;q)_k}
= \frac{-q^{-2}}{(q,q^4;q^5)_{\infty}},$$
by the Rogers--Ramanujan identity \myref{rra}. 
(Because $|q|<1$, computing the poles term-wise is justified).
This agrees with the residue of $z=q^{-1}$ on the right.
We conclude that the ratio of the left- and right-hand sides
is a constant. Taking $z\to 1$ shows that this constant is $1$,
which completes the verification of \myref{bilrrg1}. \hfill\qed

Similar verification proofs can be given for the other bilateral
identities that we obtained which involve the variable $z$.

All results in Sections~\ref{sec:bilRR} and \ref{sec:multRR}
are obtained by specializing a sole master identity,
namely \myref{bil68}.
It is feasible for an interested reader to check whether
any other relevant specializations of \myref{bil68}
were missed (in our collection of bilateral Rogers--Ramanujan
type identities) that would yield further noteworthy identities.
\end{Remark}

As consequence of Theorem~\ref{thm5},
we obtain the following four bilateral summations:

\begin{Corollary}[Bilateral modulus $25$ identities]\label{cor:bilrr}
\begin{subequations}\label{bilrr}
\begin{align}\label{bilrra}
\sum_{k=-\infty}^\infty
  \frac{q^{k(5k-3)}}{(q;q^5)_k}
  &=\frac{(q^4;q^5)_\infty(q^{10},q^{15},q^{25};q^{25})_\infty}
                               {(q^2,q^3;q^5)_\infty},\\*
\label{bilrrb}\sum_{k=-\infty}^\infty
\frac{q^{(k-1)(5k-1)}}{(q^2;q^5)_{k}}&=
\frac{(q^{3};q^5)_\infty(q^5,q^{20},q^{25};q^{25})_\infty}{(q,q^4;q^5)_\infty},\\*
\label{bilrrc}\sum_{k=-\infty}^\infty
\frac{q^{k(5k-4)}}{(q^3;q^5)_k}&=
\frac{(q^2;q^5)_\infty(q^5,q^{20},q^{25};q^{25})_\infty}{(q,q^4;q^5)_\infty},\\*
\label{bilrrd}\sum_{k=-\infty}^\infty
\frac{q^{k(5k+3)}}{(q^4;q^5)_k}&=
\frac{(q;q^5)_\infty(q^{10},q^{15},q^{25};q^{25})_\infty}{(q^2,q^3;q^5)_\infty}.
\end{align}
\end{subequations}
\end{Corollary}
Combinatorial interpretations of the left- and right-hand sides
of these identities can be given
(see the discussion in Section~\ref{sec:conc}); since
the formulations are rather lengthy, we do not provide the details here.

To deduce the bilateral identities in Corollary~\ref{cor:bilrr},
first replace $q$ by $q^5$ in \myref{bilrrg} and then observe that
the respective $z=q^{-3}$ and  $z=q^{-2}$ cases of \myref{bilrrg1}
give \myref{bilrrb} and \myref{bilrrc}, whereas
the respective $z=q^{-4}$ and $z=q^{-1}$ cases of \myref{bilrrg2}
give \myref{bilrra} and \myref{bilrrd}.

\begin{Remark}\label{rem:th}
Tim Huber has kindly informed the author how the series appearing
in Corollary~\ref{cor:bilrr} are related to weight $1/5$ modular forms
for $\Gamma(5)$.
(See \cite{H} for a theory of theta functions to the quintic base).

In particular, writing $q=e^{2\pi{\rm i}\tau}$ (where ${\rm i}^2=-1$),
the following functions $A,B$ are weight $1/5$
modular forms for $\Gamma(5)$ with a fifth root of unity as a multiplier:
\begin{subequations}
\begin{align}
  A(\tau)&=\frac{q^{1/5}(q;q)_\infty^{2/5}}{(q^2;q^3;q^5)_\infty}\notag\\*
         &=\frac{q^{1/5}(q;q)_\infty^{2/5}}{(q;q^5)_\infty
           (q^{10},q^{15};q^{25};q^{25})_\infty}
           \sum_{k=-\infty}^\infty\frac{q^{k(5k+3)}}{(q^4;q^5)_k}\notag\\*
         &=\frac{q^{1/5}(q;q)_\infty^{2/5}}{(q^4;q^5)_\infty
           (q^{10},q^{15};q^{25};q^{25})_\infty}
           \sum_{k=-\infty}^\infty\frac{q^{k(5k-3)}}{(q;q^5)_k},\\
  B(\tau)&=\frac{(q;q)_\infty^{2/5}}{(q;q^4;q^5)_\infty}\notag\\*
         &=\frac{(q;q)_\infty^{2/5}}{(q^3;q^5)_\infty
           (q^{5},q^{20};q^{25};q^{25})_\infty}
           \sum_{k=-\infty}^\infty\frac{q^{(k-1)(5k-1)}}{(q^2;q^5)_k}\notag\\*
         &=\frac{(q;q)_\infty^{2/5}}{(q^2;q^5)_\infty
           (q^{2},q^{20};q^{25};q^{25})_\infty}
  \sum_{k=-\infty}^\infty\frac{q^{k(5k-4)}}{(q^3;q^5)_k}.
\end{align}
\end{subequations}
The graded ring of modular forms of integer weight for $\Gamma_1(5)$
is generated by $A^5$ and $B^5$. Moreover, a basis for
the space of weight $1$ modular forms on $\Gamma_1(5)$
is given by
$$
A^5,\; A^4B,\; A^3 B^2,\; A^2 B^3,\; A B^4,\; B^5.
$$

It is possible to make similar connections to modularity
for some of the other series appearing in this section,
in particular, for the series in Corollary~\ref{cor:bilgg}.
\end{Remark}

Our next result is a bilateral extension of the two
G\"ollnitz--Gordon identities in \myref{gg}.

\begin{Theorem}\label{thm8}
We have the following two bilateral summations:
\begin{subequations}\label{bilggg}
\begin{align}\label{bilggg1}
\sum_{k=-\infty}^\infty\frac{(-zq;q^2)_k}{(zq^2;q^2)_k}
q^{k^2}z^k
={}&\frac{(-zq,1/z;q^2)_\infty}
{(z^2q^2,1/z^2;q^2)_\infty}(q^8;q^8)_\infty\notag\\*
&\times\left[(z^4q^5,z^{-4}q^3;q^8)_\infty+
z^{-1}(z^4q^3,z^{-4}q^5;q^8)_\infty
\right],\\*\label{bilggg2}
\sum_{k=-\infty}^\infty\frac{(-zq;q^2)_k}{(zq^2;q^2)_k}
q^{k(k+2)}z^k
={}&\frac{(-zq,1/z;q^2)_\infty}
{(z^2q^2,1/z^2;q^2)_\infty}(q^8;q^8)_\infty\notag\\*
&\times\left[(z^4q^7,z^{-4}q;q^8)_\infty+
z^{-3}(z^4q,z^{-4}q^7;q^8)_\infty
\right],
\end{align}
\end{subequations}
for complex $z$ such that $z\not\in
\{q^{-2},q^{-4},q^{-6},\ldots\}\cup\{-q^{-1},-q^{-3},-q^{-5},\ldots\}$.
\end{Theorem}

The $z\to 1$ limit of \myref{bilggg1} gives \myref{gga},
while the $z\to 1$ limit of \myref{bilggg2} gives \myref{ggb}.

As consequence of Theorem~\ref{thm8},
we obtain the following four bilateral summations:

\begin{Corollary}[Bilateral modulus $32$ identities]\label{cor:bilgg}
\begin{subequations}\label{bilgg}
\begin{align}\label{bilgga}
\sum_{k=-\infty}^\infty\frac{(-q^5;q^8)_{k}}{(q;q^8)_{k+1}}
q^{(k+2)(4k+1)}&=\frac{(q^7;q^8)_\infty(q^8;q^{16})_\infty
(q^{32};q^{32})_\infty}{(q^5,q^6;q^8)_\infty(q^2;q^{16})_\infty},\\*
\label{bilggb}\sum_{k=-\infty}^\infty\frac{(-q^7;q^8)_{k}}{(q^3;q^8)_{k+1}}
q^{k(4k+3)}&=\frac{(q^5;q^8)_\infty(q^8;q^{16})_\infty
(q^{32};q^{32})_\infty}{(q^2,q^7;q^8)_\infty(q^6;q^{16})_\infty},\\*
\label{bilggc}\sum_{k=-\infty}^\infty\frac{(-q;q^8)_k}{(q^5;q^8)_k}
q^{k(4k-3)}&=\frac{(q^3;q^8)_\infty(q^8;q^{16})_\infty
(q^{32};q^{32})_\infty}{(q,q^6;q^8)_\infty(q^{10};q^{16})_\infty},\\*
\label{bilggd}\sum_{k=-\infty}^\infty\frac{(-q^3;q^8)_k}{(q^7;q^8)_k}
q^{k(4k+7)}&=\frac{(q;q^8)_\infty(q^8;q^{16})_\infty
(q^{32};q^{32})_\infty}{(q^2,q^3;q^8)_\infty(q^{14};q^{16})_\infty}.
\end{align}
\end{subequations}
\end{Corollary}

To deduce the bilateral identities in Corollary~\ref{cor:bilgg},
first replace $q$ by $q^4$ in \myref{bilggg} and then observe that
the respective $z=q^3$ and  $z=q^{-3}$ cases of \myref{bilggg1}
give \myref{bilggb} and \myref{bilggc}, whereas
the respective $z=q$ and $z=q^{-1}$ cases of \myref{bilggg2}
give \myref{bilgga} and \myref{bilggd}.

Notice that Equations \myref{bilgga} and \myref{bilggb}
can be obtained from each other by replacing $q$ by $-q$.
The same relation also holds for Equations \myref{bilggc} and \myref{bilggd}.

We would like to stress that the bilateral summations in
Corollaries~\ref{cor:bilrr} and \ref{cor:bilgg},
which we believe to be new (and also \emph{beautiful}, in line with
Hardy's quote about \myref{rr} stated in the introduction),
are \textit{not} special cases of the following 
bilateral extension of the Lebesgue identity
\begin{equation}\label{billeb}
\sum_{k=-\infty}^\infty\frac{(a;q)_k}{(bq;q)_k}q^{\binom{k+1}2}b^k=
\frac{(q^2,abq,q/ab,bq^2/a;q^2)_\infty}{(bq,q/a;q)_\infty}
\end{equation}
(which can be obtained from \cite[Appendix~(II.30), $c\to\infty$ followed by
$(a,b)\mapsto (ab,a)$]{GR}).

A noteworthy special case of \myref{billeb} due to G\"ollnitz~\cite{Goe},
which should be compared to the G\"ollnitz--Gordon identities in \myref{gg},
is obtained by letting
$(a,b,q)\mapsto (-q,1,q^2)$:
\begin{equation}
\sum_{k=0}^\infty\frac{(-q;q^2)_k}{(q^2;q^2)_k}q^{k(k+1)}=
\frac 1{(q^2,q^3,q^7;q^8)_\infty}.
\end{equation}
Another noteworthy special case of \myref{billeb}
is obtained by letting $(a,b)\mapsto(-q,1)$:
\begin{equation}
\sum_{k=0}^\infty\frac{(-q;q)_k}{(q;q)_k}q^{\binom{k+1}2}=
\frac {(q^4;q^4)_\infty}{(q;q)_\infty}.
\end{equation}
which is identity (8) in Slater's list.

Other bilateral identities of the Rogers--Ramanujan type
are collected in the following theorem:

\begin{Theorem}\label{thm:bilcc}
We have the following four bilateral summations:
\begin{subequations}\label{bilcc}
\begin{align}
\label{bilccb}
\sum_{k=-\infty}^\infty\frac{(-z;q)_k}{(z^2q;q^2)_k}q^{\binom{k}2}z^k
={}&\frac{(-z;q)_\infty(q;q^2)_\infty}{(z^2;q)_\infty(q^2/z^2;q^2)_\infty}
(q^3,z^3,z^{-3}q^3;q^3)_\infty,\\
\label{bilcca}
\sum_{k=-\infty}^\infty\frac{(-z;q^2)_k}{(zq;q)_{2k}}
q^{k(k+1)}z^k
={}&\frac{(q/z;q)_\infty(-zq^2;q^2)_\infty}
{(z^2q^2,q^2/z^2,q;q^2)_\infty}(q^6,-z^3q^3,-z^{-3}q^3;q^6)_\infty,\\\notag
\label{bilcca2}
\sum_{k=-\infty}^\infty\frac{(-z;q^2)_k}{(z;q)_{2k}}
q^{k(k-1)}z^k
={}&\frac{(q/z;q)_\infty(-z;q^2)_\infty}
{(z^2,q^2/z^2,q;q^2)_\infty}(q^6;q^6)_\infty\\*
&\times\left[(-z^3q,-z^{-3}q^5;q^6)_\infty+
z^2(-z^3q^5,-z^{-3}q;q^6)_\infty
\right],\\\notag
\label{biljac}
\sum_{k=-\infty}^\infty\frac{q^{2k^2}}{(z;q)_{2k+1}}z^{2k}
={}&\frac{(q/z;q)_\infty}
{(z^2,q^2/z^2,q;q^2)_\infty}(q^8;q^8)_\infty\\*
&\times\left[(-z^4q^3,-z^{-4}q^5;q^8)_\infty+
z(-z^4q^5,-z^{-4}q^3;q^8)_\infty
\right],
\end{align}
\end{subequations}
for complex $z$ such that the series on the left-hand sides
have no poles.
\end{Theorem}

The case $q\mapsto q^2$, followed by $z\to q$, of \myref{bilccb}
reduces to identity (25) in Slater's list, which can be stated as
\begin{equation}\label{s3}
\sum_{k=0}^\infty\frac{(-q;q^2)_k}{(q^4;q^4)_k}q^{k^2}=
\frac{(q^2;q^2)_\infty(q^3;q^3)_\infty^2}
{(q;q)_\infty(q^4;q^4)_\infty(q^6;q^6)_\infty}=
\frac{(q^3;q^3)_\infty}{(q^4;q^4)_\infty(q,q^5;q^6)_\infty}.
\end{equation}
The $z\to 1$ case of \myref{bilcca}
reduces to identity (48) in Slater's list, which is
\begin{equation}\label{s1}
\sum_{k=0}^\infty\frac{(-1;q^2)_k}{(q;q)_{2k}}q^{k(k+1)}=
\frac{(q^4;q^4)_\infty(q^6;q^6)_\infty^5}
{(q^2;q^2)_\infty^2(q^3;q^3)_\infty^2(q^{12};q^{12})_\infty^2}.
\end{equation}
The $z\to q$ cases of \myref{bilcca} and  \myref{bilcca2}
reduce to identities (50) and (29) in Slater's list, namely
\begin{subequations}
\begin{align}\label{s1b}
\sum_{k=0}^\infty\frac{(-q;q^2)_k}{(q;q)_{2k+1}}q^{k(k+2)}&=
\frac{(q^2;q^2)_\infty(q^{12};q^{12})_\infty^2}
{(q;q)_\infty(q^4;q^4)_\infty(q^6;q^6)_\infty},\\*
\intertext{and}\label{s12}
\sum_{k=0}^\infty\frac{(-q;q^2)_k}{(q;q)_{2k}}q^{k^2}&=
\frac{(q^6;q^6)_\infty^2}
{(q;q)_\infty(q^{12};q^{12})_\infty},
\end{align}
\end{subequations}
respectively. The $z\to q^2$ case of \myref{bilcca2}
reduces to identity (28) in Slater's list, i.e.,
\begin{equation}\label{s12b}
\sum_{k=0}^\infty\frac{(-q^2;q^2)_k}{(q;q)_{2k+1}}q^{k(k+1)}=
\frac{(q^3;q^3)_\infty(q^{12};q^{12})_\infty}
{(q;q)_\infty(q^6;q^6)_\infty}.
\end{equation}
Multiplication of both sides of \myref{biljac} by $(1-z)$ and letting
$z\to 1$ reduces to a sum by F.H.~Jackson~\cite{Ja},
also given by Slater as identity (39), which is
\begin{equation}\label{jack1}
\sum_{k\ge 0}\frac{q^{2k^2}}{(q;q)_{2k}}=
\frac1{(q^3,q^4,q^5;q^8)_\infty(q^2,q^{14};q^{16})_\infty}.
\end{equation}
The $z\to q$ case of \myref{biljac} reduces to identity (38)
in Slater's list, namely
\begin{equation}\label{jack2}
\sum_{k\ge 0}\frac{q^{2k(k+1)}}{(q;q)_{2k+1}}=
\frac1{(q,q^4,q^7;q^8)_\infty(q^6,q^{10};q^{16})_\infty}.
\end{equation}
The $z\to -1$ cases of \myref{bilcca} and \myref{bilcca2}
reduce, after replacing the summation index $k$ by $-k$, to the identities
\begin{subequations}\label{s2s}
\begin{align}\label{s2}
\sum_{k\ge 0}\frac{(-1;q)_{2k}}{(q^2;q^2)_k}q^k&=
\frac{(q^2;q^2)_\infty(q^3;q^3)_\infty^2}
{(q;q)_\infty^2(q^6;q^6)_\infty}
=\frac1{(q,q^2;q^3)_\infty(q,q^5;q^6)_\infty},\\*\label{s22}
\sum_{k\ge 0}\frac{(-q;q)_{2k}}{(q^2;q^2)_k}q^k&=
\frac{(q^6;q^6)_\infty^2}
{(q;q)_\infty(q^3;q^3)_\infty}
=\frac1{(q,q^2;q^3)_\infty(q^3,q^3;q^6)_\infty}.
\end{align}
\end{subequations}
Equation \myref{s2} is given by Slater as identity (24), while
\myref{s22} is due to Ismail and Stanton~\cite[Thm.~7]{IS}.
It is not difficult to transform the ${}_2\phi_1$ series
(with vanishing lower parameter)
on the left-hand sides of Equations \myref{s2s} by suitable instances of
the $q$-Pfaff transformation \cite[Appendix~(III.4)]{GR} to
${}_1\phi_1$ series, by which \myref{s2}
is seen to be equivalent to the $q\mapsto-q$ case of \myref{s1}
and also to \myref{s3}, while \myref{s22} is then seen to be equivalent to
an identity by Ramanujan (cf.\ \cite[p.~87, Entry 4.2.11]{AB})
and also to the $q\mapsto-q$ case of \myref{s12b}.

As consequence of Equation \myref{bilcca2}, we obtain the following two
bilateral summations:

\begin{Corollary}[Bilateral modulus $6$ identities]\label{cor:bilcs}
\begin{subequations}\label{bilcs}
\begin{align}\label{bilcs1}
\sum_{k=-\infty}^\infty\frac{(q^5;q^6)_{k}}{(-q^2;q^3)_{2k+1}}(-1)^k
q^{k(3k+2)}&=
\frac{(q^5,q^6;q^6)_\infty}{(q;q^3)_\infty(q^3,q^4;q^6)_\infty},\\*
\label{bilcs2}
\sum_{k=-\infty}^\infty\frac{(q;q^6)_{k}}{(-q;q^3)_{2k}}(-1)^k
q^{k(3k-2)}&=
\frac{(q,q^6;q^6)_\infty}{(q^2;q^3)_\infty(q^2,q^3;q^6)_\infty}.
\end{align}
\end{subequations}
\end{Corollary}

To deduce the bilateral identities in Corollary~\ref{cor:bilcs},
first replace $q$ by $q^3$ in \myref{bilcca2} and then observe that
the respective $z=-q^{-1}$ and  $z=-q$ cases give \myref{bilcs1} and
\myref{bilcs2}.

The identities in Corollary~\ref{cor:bilcs} become even nicer
if the summation index $k$ is replaced by $-k$:
\begin{manualCorollary}{\ref{cor:bilcs}$'$}[Bilateral modulus $6$ identities]
\label{cor:bilcsp}
\begin{subequations}\label{bilcsp}
\begin{align}\label{bilcs1p}
\sum_{k=-\infty}^\infty\frac{(-q;q^3)_{2k-1}}{(q;q^6)_{k}}q^{3k-2}
&=\frac{(q^5,q^6;q^6)_\infty}{(q;q^3)_\infty(q^3,q^4;q^6)_\infty},\\*
\label{bilcs2p}
\sum_{k=-\infty}^\infty\frac{(-q^2;q^3)_{2k}}{(q^5;q^6)_k}q^{3k}
&=\frac{(q,q^6;q^6)_\infty}{(q^2;q^3)_\infty(q^2,q^3;q^6)_\infty}.
\end{align}
\end{subequations}
\end{manualCorollary}

Further, as consequence of Equation \myref{biljac},
we obtain the following four bilateral summations:

\begin{Corollary}[Bilateral modulus $32$ identities]\label{cor:biljacs}
\begin{subequations}\label{biljacs}
\begin{align}
\label{biljacs1}
\sum_{k=-\infty}^\infty\frac{q^{2k(4k-3)}}{(q;q^8)_k(-q^5;q^8)_k}&=
\frac{(q^4,q^7;q^8)_\infty(q^{32};q^{32})_\infty}{(q^2,q^3;q^8)_\infty
(q^{14};q^{16})_\infty},\\*\label{biljacs3}
\sum_{k=-\infty}^\infty\frac{q^{2k(4k+3)}}{(q^3;q^8)_{k+1}(-q^7;q^8)_k}&=
\frac{(q^4,q^5;q^8)_\infty(q^{32};q^{32})_\infty}{(q,q^6;q^8)_\infty
(q^{10};q^{16})_\infty},\\*\label{biljacs5}
\sum_{k=-\infty}^\infty\frac{q^{2k(4k-3)}}{(q^5;q^8)_k(-q;q^8)_k}&=
\frac{(q^3,q^4;q^8)_\infty(q^{32};q^{32})_\infty}{(q^2,q^7;q^8)_\infty
(q^6;q^{16})_\infty},\\*\label{biljacs7}
\sum_{k=-\infty}^\infty\frac{q^{2k(4k+3)}}{(q^7;q^8)_{k}(-q^3;q^8)_{k+1}}&=
\frac{(q,q^4;q^8)_\infty(q^{32};q^{32})_\infty}{(q^5,q^6;q^8)_\infty
(q^2;q^{16})_\infty}.
\end{align}
\end{subequations}
\end{Corollary}

To deduce the bilateral identities in Corollary~\ref{cor:biljacs},
first replace $q$ by $-q^4$ in \myref{biljac} and then observe that
the respective $z=q^3$ and  $z=q^{-3}$ cases give
\myref{biljacs3} and \myref{biljacs5}. The identities in
\myref{biljacs1} and \myref{biljacs7} follow by replacing
$q$ by $-q$ in \myref{biljacs5} and \myref{biljacs3}, respectively.

\section{Derivations of the main results
in the single series case}\label{sec:bilRR}

A rich source of material on basic hypergeometric series
is Gasper and Rahman's classic textbook \cite{GR}.
In particular, we refer to that book
for standard notions (such as that of a bilateral basic hypergeometric 
${}_r\psi_s$ series), and to Appendix I of that book for the
elementary manipulations of $q$-shifted factorials, which we employ
without explicit mention.

An identity which we make crucial use of is Jacobi's triple product identity
(cf.\ \cite[(II.28)]{GR})
\begin{equation}\label{jtpi}
\sum_{k=-\infty}^\infty q^{\binom k2}(-z)^k=(q,z,q/z;q)_\infty.
\end{equation}


Our starting point for deriving bilateral summations of the
Rogers--Ramanujan type
is the following transformation of a general
bilateral $_2\psi_2$ series into a multiple of a very-well-poised
$_6\psi_8$ series due to Bailey~\cite[(3.2)]{Ba} (see also
\cite[Exercise~5.11, second identity]{GR}).
\begin{align}\label{bil68}\notag
\sum_{k=-\infty}^\infty\frac{(e,f;q)_k}{(aq/c,aq/d;q)_k}
\left(\frac{aq}{ef}\right)^k
=\frac{(q/c,q/d,aq/e,aq/f;q)_\infty}
{(aq,q/a,aq/cd,aq/ef;q)_\infty}&\\*\times
\sum_{k=-\infty}^\infty\frac{(1-aq^{2k})(c,d,e,f;q)_k}
{(1-a)(aq/c,aq/d,aq/e,aq/f;q)_k}q^{k^2}\left(\frac{a^3q}{cdef}\right)^k&,
\end{align}
valid for $|aq/cd|<1$ and $|aq/ef|<1$. Bailey obtained this
transformation by bilateralizing Watson's transformation
(cf.\ \cite[(III.18)]{GR}) using the same method (replacing
$n$ by $2n$, shifting the summation index $k\mapsto k+n$, suitably
shifting parameters and taking the limit $n\to\infty$),
applied by Cauchy~\cite{Cau} in his second proof of
Jacobi's triple product identity.


In \myref{bil68} we now let $f\to\infty$ and
perform the simultaneous
substitutions $(a,c,d,e)\mapsto(az,az/b,az/c,a)$.
This yields the following transformation of a general $_1\psi_2$ series
into a multiple of a very-well-poised $_5\psi_8$ series.
\begin{align}\notag
\sum_{k=-\infty}^\infty\frac{(a;q)_k}{(bq,cq;q)_k}q^{\binom{k+1}2}(-z)^k
=\frac{(bq/az,cq/az,zq;q)_\infty}
{(azq,q/az,bcq/az;q)_\infty}&\\*\times
\sum_{k=-\infty}^\infty\frac{(1-azq^{2k})(az/b,az/c,a;q)_k}
{(1-az)(bq,cq,zq;q)_k}q^{3\binom k2}\left(-bczq^2\right)^k&,
\label{bil58}
\end{align}
valid for $|bcq/az|<1$.

Theorems~\ref{thm5}, \ref{thm8} and \ref{thm:bilcc}
all appear by combining special instances of
\myref{bil58} with \myref{jtpi}.

\begin{proof}[Proof of Theorem~\ref{thm5}]
In \myref{bil58}, we first let $c\to 0$, perform the
substitutions $(b,z)\mapsto (z,bz/a)$ and let $a\to\infty$.
We obtain
\begin{equation}\label{c0}
\sum_{k=-\infty}^\infty\frac{q^{k^2}\,(bz)^k}
{(zq;q)_k}
=\frac{(q/b;q)_\infty}
{(bzq,q/bz;q)_\infty}
\sum_{k=-\infty}^\infty\frac{(1-bzq^{2k})}{(1-bz)}\frac{(b;q)_k}
{(zq;q)_k}q^{5\binom k2}\left(-b^2z^3q^2\right)^k.
\end{equation}
The $b=z$ case of \myref{c0} reduces to
\begin{equation*}
\sum_{k=-\infty}^\infty\frac{q^{k^2}z^{2k}}
{(zq;q)_k}
=\frac{(1/z;q)_\infty\,z^{-1}}
{(z^2q,1/z^2;q)_\infty}
\sum_{k=-\infty}^\infty(1+zq^{k})
q^{5\binom k2}\left(-z^5q^2\right)^k,
\end{equation*}
which after two applications of \myref{jtpi} yields \myref{bilrrg1}.
Similarly, the $b=zq$ case of \myref{c0} reduces to
\begin{equation*}
\sum_{k=-\infty}^\infty\frac{q^{k(k+1)}z^{2k}}
{(zq;q)_k}
=\frac{(1/z;q)_\infty}
{(z^2q,1/z^2;q)_\infty}
\sum_{k=-\infty}^\infty(1-z^2q^{1+2k})
q^{5\binom k2}\left(-z^5q^4\right)^k,
\end{equation*}
which after two applications of \myref{jtpi} yields \myref{bilrrg2}.
\end{proof}

In the remaining proofs we only sketch the most relevant steps.
\begin{proof}[Proof of Theorem~\ref{thm8}]
In \myref{bil58}, we first let $c\to 0$,
replace $q$ by $q^2$ and set $(a,b,z)\mapsto(-zq,z,-zq^{-1})$.
The result, after two applications of  \myref{jtpi},
is \myref{bilggg1}. Now \myref{bilggg2} can readily be obtained from
\myref{bilggg1} by replacing $z$ by $-1/zq$ and reversing the sum.
\end{proof}

\begin{proof}[Proof of Theorem~\ref{thm:bilcc}]
The identity \myref{bilccb} follows from \myref{bil58} by making
the substitution $(a,b,c,z)\mapsto(-z,zq^{-1/2},-zq^{-1/2},-z)$,
and applying \myref{jtpi}.
The identity \myref{bilcca} follows from \myref{bil58}
by replacing $q$ by $q^2$, setting
$(a,b,c,z)\mapsto(-z,z,zq^{-1},-z)$, and applying \myref{jtpi}.
The identity \myref{bilcca2} follows from \myref{bil58}
by replacing $q$ by $q^2$, setting
$(a,b,c,z)\mapsto(-z,zq^{-1},zq^{-2},-zq^{-2})$,
and applying \myref{jtpi} twice.
The identity \myref{biljac} follows from \myref{bil58}
by replacing $q$ by $q^2$, setting
$(b,c,z)\mapsto(z,zq^{-1},z^2/a)$ followed by taking $a\to\infty$,
applying \myref{jtpi} twice and dividing both sides by $(1-z)$.
\end{proof}

\section{Multisum extensions}
\label{sec:multRR}

Here
we derive multisum extensions of the results from Section~\ref{sec:thms}.
Throughout we assume $r\ge 2$. We write $k=(k_1,\dots,k_{r-1})$ 
and define $k_r:=0$. Further, we define
\begin{equation*}
\Lambda^{(r-1)}:=\{k\in\mathbb Z^{r-1}\mid
\infty>k_1\ge\dots\ge k_{r-1}>-\infty\}
\end{equation*}
in order to compactly specify the range of our multilateral summations.

Our multisum extension of Theorem~\ref{thm5} is a
multilateral extension of the Andrews--Gordon identities \cite{An74},
which, for integers
$r$ and $i$ with $r\ge 2$ and $1\le i\le r$,
can be written as
\begin{equation}\label{rrmult}
\sum_{k\in\Lambda^{(r-1)}}
\frac{q^{\sum_{j=1}^{r-1} k_j^2+\sum_{j=i}^{r-1} k_j}}
{(q;q)_{k_{r-1}}\prod_{j=1}^{r-2}(q;q)_{k_j-k_{j+1}}}
=\frac {(q^i,q^{2r+1-i},q^{2r+1};q^{2r+1})_\infty}{(q;q)_\infty}.
\end{equation}
An elementary, inductive proof of \myref{rrmult} was given by
Andrews in \cite{An84}.

Our multisum extension of Theorem~\ref{thm8}
is a multilateral extension of Andrews generalized
G\"ollnitz--Gordon identity \cite{An75,An} (the $i=r$ case
of identity \myref{ggmult} below), which has been extended
to a family of identities for integers $r$ and $i$ with
$r\ge 2$ and $1\le i\le r$, by Bressoud~\cite{Br2}.
It can be written as
\begin{equation}\label{ggmult}
\sum_{k\in\Lambda^{(r-1)}}
\frac{(-q^{1-k_1};q^2)_{k_1}
q^{2\sum_{j=1}^{r-1} k_j^2+2\sum_{j=i}^{r-1} k_j}}
{(q^2;q^2)_{k_{r-1}}\prod_{j=1}^{r-2}(q^2;q^2)_{k_j-k_{j+1}}}
=\frac {(q^2;q^4)_\infty
(q^{2i-1},q^{4r+1-2i},q^{4r};q^{4r})_\infty}
{(q;q)_\infty}.
\end{equation}
The identities in \myref{rrmult} and \myref{ggmult}
reduce to \myref{rr} and \myref{gg}, respectively, for $r=2$.

In \cite{Br}, Bressoud also gave an even modulus analogue of the
Andrews--Gordon identities in \myref{rrmult}, namely
\begin{equation}\label{rr2mult}
\sum_{k\in\Lambda^{(r-1)}}
\frac{q^{\sum_{j=1}^{r-1} k_j^2+\sum_{j=i}^{r-1} k_j}}
{(q^2;q^2)_{k_{r-1}}\prod_{j=1}^{r-2}(q;q)_{k_j-k_{j+1}}}
=\frac {(q^i,q^{2r-i},q^{2r};q^{2r})_\infty}{(q;q)_\infty},
\end{equation}
where $1\le i\le r$.
The $(r,i)=(2,1)$ and $(r,i)=(2,2)$ cases of \myref{rr2mult} are
special cases of the $q$-binomial theorem.

Our multilateral summations in this section are obtained by a procedure
completely analogous to the single series case. Our starting point is the
following multisum transformation which follows from a result by
Agarwal, Andrews and
Bressoud~\cite[Theorem~3.1 with Equations~(4.1) and (4.2)]{AAB}:
\begin{Proposition}
Let $r$ and $i$ be integers with $r\ge 2$ and $1\le i\le r$.
Further, let $n$ be a nonnegative integer.
Then, with $k_0:=n$ and $k_r:=0$,
we have the following series transformation:
\begin{align}\label{aab}
\sum_{n\ge k_1\ge\cdots\ge k_{r-1}\ge 0}\Bigg(&
\prod_{j=1}^{r}\frac{(b_j,c_j;q)_{k_j}}{(q;q)_{k_{j-1}-k_j}}
\prod_{j=1}^{i-1}\frac{(a/b_jc_j;q)_{k_{j-1}-k_j}}
{(a/b_j,a/c_j;q)_{k_{j-1}}}\notag\\*&\times
\prod_{j=i}^{r}\frac{(aq/b_jc_j;q)_{k_{j-1}-k_j}}{(aq/b_j,aq/c_j;q)_{k_{j-1}}}
\prod_{j=1}^{r-1}\left(\frac a{b_jc_j}\right)^{k_j}\cdot
q^{\sum_{j=i}^{r-1}k_i}\Bigg)\notag\\
=\sum_{k=0}^n\Bigg(&\frac{(a;q)_k(-1)^k
q^{\binom k2+(r+1-i)k}}{(q;q)_k(q;q)_{n-k}(a;q)_{n+k}}
\frac{a^{rk}}{\prod_{j=1}^{r}(b_jc_j)^k}\notag\\*
&\times\prod_{j=1}^{i-1}\frac{(b_j,c_j;q)_k}{(a/b_j,a/c_j;q)_k}
\prod_{j=i}^{r}\frac{(b_j,c_j;q)_k}{(aq/b_j,aq/c_j;q)_k}\notag\\*
&\times\left[1
+\frac{(1-q^k)a q^{k-1}}{(1-aq^{k-1})}\prod_{j=i}^{r}
\frac{b_jc_j\,(1-aq^k/b_j)(1-aq^k/c_j)}{aq\,(1-b_jq^{k-1})(1-c_jq^{k-1})}
\right]\Bigg).
\end{align}
\end{Proposition}

This is (even in the $i=r$ case) different from the multivariate
Watson transformation due to Andrews~\cite[Thm.~4]{An75}
(which, if used as a starting point instead,
would only serve to prove the extremal $i=1$ and $i=r$ cases
of the multisum identities we are after).

By multilateralization, we now deduce the following
transformation of multisums.
\begin{Corollary}\label{cor68}
Assuming $k_0:=\infty$, we have for
$r\ge 2$ and $1\le i\le r$ the following transformation:
\begin{align}\label{mult68}\notag
\sum_{\kL}\Bigg(&
\frac{\prod_{j=1}^{r-1}(b_j,c_j;q)_{k_j}}{\prod_{j=1}^{r-2}(q;q)_{k_j-k_{j+1}}}
\prod_{j=1}^{i-1}\frac{(a/b_jc_j;q)_{k_{j-1}-k_j}}
{(a/b_j,a/c_j;q)_{k_{j-1}}}\\*&\times
\frac{\prod_{j=i}^r(aq/b_jc_j;q)_{k_{j-1}-k_j}}
{\prod_{j=i}^{r}(aq/b_j,aq/c_j;q)_{k_{j-1}}}
\prod_{j=1}^{r-1}
\left(\frac{a}{b_jc_j}\right)^{k_j}\cdot q^{\sum_{j=i}^{r-1}k_j}\Bigg)\notag\\
={}&\frac{(q/b_r,q/c_r;q)_\infty}
{(a,q/a,aq/b_rc_r;q)_\infty}\notag\\*\times
\sum_{k=-\infty}^\infty \Bigg(&q^{k^2+(r-i)k}\frac{a^{(r+1)k}}
{\prod_{j=1}^{r}(b_jc_j)^k}
\prod_{j=1}^{i-1}\frac{(b_j,c_j;q)_k}
{(a/b_j,a/c_j;q)_k}\prod_{j=i}^{r}\frac{(b_j,c_j;q)_k}
{(aq/b_j,aq/c_j;q)_k}\notag\\*&\times
\left[1-\prod_{j=i}^{r}\frac{b_jc_j(1-aq^k/b_j)(1-aq^k/c_j)}
{aq(1-b_jq^{k-1})(1-c_jq^{k-1})}\right]\Bigg),
\end{align}
valid for $\big|q^{r-i}\prod_{j=1}^{r-1}(a/b_jc_j)\big|<1$ and
$\big|q^{r-i}\prod_{j=1}^{r-1}(a/b_{j+1}c_{j+1})\big|<1$.
\end{Corollary}

\begin{Remark}\label{rem}
Notice that for $i=r$ the expression in square brackets on the
right-hand side of \myref{mult68} simplifies to
\begin{equation*}
1-
\frac{bc(1-aq^k/b)(1-aq^k/c)}
{aq(1-bq^{k-1})(1-cq^{k-1})}
=\frac{(1-aq^{2k-1})(1-bc/aq)}
{(1-bq^{k-1})(1-cq^{k-1})}
\end{equation*}
(whereas the corresponding larger expression in square brackets
on the right-hand side of \myref{aab} does not factorize for $i=r$),
where we replaced $(b_{r+1},c_{r+1})$ by $(b,c)$,
and the transformation in \myref{mult68}
is then seen to be an $(r-1)$-dimensional generalization
of the bilateral transformation in \myref{bil68}
(with $a$ replaced by $a/q$) which alternatively could
also be obtained by multilateralization of Andrews' formula
\cite[Thm.~4]{An75}.
\end{Remark}

\begin{proof}[Proof of Corollary~\ref{cor68}]
To obtain \myref{mult68} from \myref{aab}, replace $n$ by $2n$,
shift the summation indices $k_1,\dots,k_{r-1}$ (on the left-hand side) and
$k$ (on the right-hand side) by $n$, perform the substitutions
$a\mapsto aq^{-2n}$, $b_j\mapsto b_jq^{-n}$, $c_j\mapsto c_jq^{-n}$,
for $j=1,\dots,r$, and let $n\to\infty$ while appealing to
Tannery's theorem for taking termwise limits.
\end{proof}

All the multisum identities of Rogers--Ramanujan type
in this section are derived by means of the following lemma,
which extends Equation~\myref{bil58}:
\begin{Lemma}\label{lem}
We have for $r\ge 2$ and $1\le i\le r$ the following transformation:
\begin{align}\label{mult58}\notag
\sum_{\kL}&
\frac{(q^{1-k_1}/a;q)_{k_1}\,q^{\sum_{j=1}^{r-1}k_j^2+\sum_{j=i}^{r-1}k_j}}
{(bq,cq;q)_{k_{r-1}}\prod_{j=1}^{r-2}(q;q)_{k_j-k_{j+1}}}
(az)^{\sum_{j=1}^{r-1}k_j}\\
&=\frac{(zq,b/az,c/az;q)_\infty}
{(azq,1/az,bc/az;q)_\infty}\notag\\*\times
\sum_{k=-\infty}^\infty\bigg(&
q^{(2r-1)\binom k2}\left(-a^{r-2}bcz^rq^{2r-i}\right)^k
\frac{(a,azq/b,azq/c;q)_k}{(zq,bq,cq;q)_k}
\notag\\*&\times
\left[1-\frac{a^{i+1-r}z^{i+1-r}q^{2(i-r)k}(1-bq^k)(1-cq^k)}
{b\,c\,(1-azq^k/b)(1-azq^k/c)}\right]\bigg).
\end{align}
\end{Lemma}

\begin{proof}
In Corollary~\ref{cor68} successively let $b_2,\dots,b_{r-1}\to\infty$ and
$c_1,c_2,\dots,c_{r-1}\to\infty$, and  perform the substitution
$(a,b_1,b_r,c_r)\mapsto(azq,a,azq/b,azq/c)$. This establishes
(together with some elementary manipulations of $q$-shifted factorials)
the $i=2,\dots,r$ cases of the Lemma.
The $i=1$ case can be established as follows:
Start with the $i=1$ case of the right-hand side of
\myref{mult58} and split the sum according to the two terms in brackets.
After shifting the summation index $k$ by one in the second sum,
the two sums can be combined and the resulting expression
is seen to be equal to the $i=r$ and $z\mapsto zq$ case of the right-hand
side of \myref{mult58}.
(For $i=r$ the expression in brackets factorizes as we know
from Remark~\ref{rem}.)
Thus the sum equals the left-hand side of the $i=r$ and $z\mapsto zq$
case of \myref{mult58} which is the same as its $i=1$ case with $z$
left unchanged.
\end{proof}

For convenience, we state the $i=r$ case of Lemma~\ref{lem}
separately:
\begin{Lemma}\label{lem2}
We have for $r\ge 2$ the following transformation:
\begin{align}\label{mult582}\notag
\sum_{\kL}
\frac{(q^{1-k_1}/a;q)_{k_1}\,q^{\sum_{j=1}^{r-1}k_j^2}
(az)^{\sum_{j=1}^{r-1}k_j}}
{(bq,cq;q)_{k_{r-1}}\prod_{j=1}^{r-2}(q;q)_{k_j-k_{j+1}}}
=\frac{(zq,bq/az,cq/az;q)_\infty}
{(azq,q/az,bcq/az;q)_\infty}&\\*\times
\sum_{k=-\infty}^\infty\frac{(1-azq^{2k})}{(1-az)}
\frac{(a,az/b,az/c;q)_k}{(zq,bq,cq;q)_k}
q^{(2r-1)\binom k2}\left(-a^{r-2}bcz^{r-1}q^r\right)^k&.
\end{align}
\end{Lemma}

From Lemma~\ref{lem} (and its special case of Lemma~\ref{lem2})
we readily deduce a number of multilateral identities of
the Rogers--Ramanujan type.

We start with a multisum generalization of Theorem~\ref{thm5}.
\begin{Theorem}\label{thm5mult}
We have for $r\ge 2$ and $1\le i\le r$ the following multilateral summations:
\begin{align}\label{bilrrgmult}
&\sum_{\kL}
\frac{q^{\sum_{j=1}^{r-1}k_j^2+\sum_{j=i}^{r-1}k_j}
\,z^{2\sum_{j=1}^{r-1}k_j}}
{(zq;q)_{k_{r-1}}\prod_{j=1}^{r-2}(q;q)_{k_j-k_{j+1}}}
=\frac{(1/z;q)_\infty}
{(1/z^2,z^2q;q)_\infty}(q^{2r+1};q^{2r+1})_\infty\notag\\*
&\times
\left[(z^{2r+1}q^{2r+1-i},z^{-2r-1}q^i;q^{2r+1})_\infty
+z^{2i-1-2r}(z^{2r+1}q^i,z^{-2r-1}q^{2r+1-i};q^{2r+1})_\infty\right],
\end{align}
for complex $z$ such that $z\not\in\{q^{-1},q^{-2},\ldots\}$.
\end{Theorem}
\begin{proof}
In Lemma~\ref{lem}, first let $c\to 0$, then perform
the substitution $(b,z)\mapsto(z,z^2/a)$ and let $a\to\infty$.
After two applications of \myref{jtpi}
the identity \myref{bilrrgmult} is obtained.
\end{proof}

The $z\to 1$ limit of \myref{bilrrgmult} reduces to Andrews--Gordon
identities in \myref{rrmult}.
As an immediate consequence of Theorem~\ref{thm5mult} we obtain
the following multisum generalization of Corollary~\ref{cor:bilrr}:
\begin{Corollary}\label{cor:multrr}
We have for $r\ge 2$ and $1\le i\le r$ the following multilateral summations:
\begin{subequations}\label{multrr}
\begin{align}\label{multrra}
&\sum_{\kL}
\frac{q^{(2r+1)\sum_{j=1}^{r-1}k_j^2-2i\sum_{j=1}^{r-1}k_j
+(2r+1)\sum_{j=i}^{r-1}k_j}}
{(q^{2r+1-i};q^{2r+1})_{k_{r-1}}
\prod_{j=1}^{r-2}(q^{2r+1};q^{2r+1})_{k_j-k_{j+1}}}\notag\\*
&=\frac{(q^i;q^{2r+1})_\infty
(q^{2i(2r+1)},q^{(2r+1-2i)(2r+3)},q^{(2r+1)^2};q^{(2r+1)^2})_\infty}
{(q^{2i},q^{2r+1-2i};q^{2r+1})_\infty},\\
\intertext{and}\label{multrrb}
&q^{(i-1)(2r+1-2i)}\sum_{\kL}
\frac{q^{(2r+1)\sum_{j=1}^{r-1}k_j^2
-2(2r+1-i)\sum_{j=1}^{r-1}k_j+(2r+1)\sum_{j=i}^{r-1}k_j}}
{(q^i;q^{2r+1})_{k_{r-1}}\prod_{j=1}^{r-2}(q^{2r+1};q^{2r+1})_{k_j-k_{j+1}}}\notag\\*
&=\frac{(q^{2r+1-i};q^{2r+1})_\infty
(q^{2i(2r+1)},q^{(2r+1-2i)(2r+1)},q^{(2r+1)^2};q^{(2r+1)^2})_\infty}
{(q^{2i},q^{2r+1-2i};q^{2r+1})_\infty}.
\end{align}
\end{subequations}
\end{Corollary}
\begin{proof}
First replace $q$ by $q^{2r+1}$ in \myref{thm5mult}.
Then the special case $z=q^{-i}$ gives \myref{multrra}, while
the special case $z=q^{i}$, after some elementary manipulations
(including a simultaneous shift of the summation indices by $-1$),
gives \myref{multrrb}.
\end{proof}

Next we give a multisum generalization of Theorem~\ref{thm8}.
\begin{Theorem}\label{thm8mult}
We have for $r\ge 2$ and $1\le i\le r$ the following multilateral summations:
\begin{align}\label{bilgggmult}
\sum_{\kL}&
\frac{(-q^{1-2k_1}/z;q^2)_{k_1}\,q^{2\sum_{j=1}^{r-1}k_j^2
+2\sum_{j=i}^{r-1}k_j}}
{(zq^2;q^2)_{k_{r-1}}\prod_{j=1}^{r-2}(q^2;q^2)_{k_j-k_{j+1}}}
z^{2\sum_{j=1}^{r-1}k_j}
=\frac{(-zq,1/z;q^2)_\infty}
{(z^2q^2,1/z^2;q^2)_\infty}
(q^{4r};q^{4r})_\infty\notag\\*
&\times\Big[(z^{2r}q^{4r+1-2i},z^{-2r}q^{2i-1};q^{4r})_\infty
+z^{2i-1-2r}(z^{2r}q^{2i-1},z^{-2r}q^{4r+1-2i};q^{4r})_\infty\Big],
\end{align}
for complex $z$ such that
$z\not\in\{q^{-2},q^{-4},q^{-6},\ldots\}\cup\{-q,-q^3,-q^5\ldots\}$.
\end{Theorem}
\begin{proof}
In Lemma~\ref{lem}, first let $c\to 0$, replace $q$ by $q^2$ and set
$(a,b,z)\mapsto(-zq,z,-zq^{-1})$.
After two applications of \myref{jtpi} the
identity \myref{bilgggmult} is obtained.
\end{proof}

The $z\to 1$ limit of \myref{bilgggmult} reduces to the Andrews--Bressoud
generalization of the G\"ollnitz--Gordon identities in \myref{ggmult}.
As an immediate consequence of Theorem~\ref{thm8mult} we obtain
the following multisum generalization of Corollary~\ref{cor:bilgg}:
\begin{Corollary}\label{cor:multgg}
We have for $r\ge 2$ and $1\le i\le r$ the following multilateral summations:
\begin{subequations}\label{multgg}
\begin{align}\label{multgga}
\sum_{\kL}&
\frac{(-q^{2r-1+2i-4rk_1};q^{4r})_{k_1}\,
q^{4r\sum_{j=1}^{r-1}k_j^2-2(2i-1)\sum_{j=1}^{r-1}k_j
+4r\sum_{j=i}^{r-1}k_j}}
{(q^{4r+1-2i};q^{4r})_{k_r}\prod_{j=1}^{r-2}
(q^{4r};q^{4r})_{k_j-k_{j+1}}}\notag\\
&=\frac{(q^{2i-1};q^{4r})_\infty
(q^{4r(2i-1)},q^{4r(2r+1-2i)},q^{8r^2};q^{8r^2})_\infty}
{(q^{2(2i-1)},q^{2r+1-2i};q^{4r})_\infty(q^{2(4r+1-2i)};q^{8r})_\infty},
\end{align}
\textit{and}
\begin{align}\label{multggb}
q^{2(2i-1)(r-i)}\sum_{\kL}&
\frac{(-q^{2r+1-2i-4rk_1};q^{4r})_{k_1}
q^{4r\sum_{j=1}^{r-1}k_j^2+2(2i-1)\sum_{j=1}^{r-1}k_j
+4r\sum_{j=i}^{r-1}k_j}}
{(q^{2i-1};q^{4r})_{1+k_{r-1}}\prod_{j=1}^{r-2}
(q^{4r};q^{4r})_{k_j-k_{j+1}}}
\notag\\
&=\frac{(q^{4r+1-2i};q^{4r})_\infty
(q^{4r(2i-1)},q^{4r(2r+1-2i)},q^{8r^2};q^{8r^2})_\infty}
{(q^{2r-1+2i},q^{2(2r+1-2i)};q^{4r})_\infty(q^{2(2i-1)};q^{8r})_\infty}.
\end{align}
\end{subequations}
\end{Corollary}
\begin{proof}
First replace $q$ by $q^{2r}$ in \myref{thm8mult}.
Then the special case $z=q^{1-2i}$ gives \myref{multgga}, while
the special case $z=q^{2i-1}$, after some elementary manipulations,
gives \myref{multggb}.
\end{proof}

Next we give a multilateral extension of the extremal $i=r$ case
of Bressoud's even modulus analogue of the
Andrews--Gordon identities \myref{rr2mult}.
(For $1\le i\le r-1$ the corresponding multilateral series do
not absolutely converge. This problem already occurs in the $r=2$ case.)
\begin{Theorem}\label{thm:rr2multlat}
We have for $r\ge 2$ the following multilateral summation:
\begin{equation}\label{rr2multlat}
\sum_{\kL}
\frac{q^{\sum_{j=1}^{r-1}k_j^2}\,z^{\sum_{j=1}^{r-1} k_j}}
{(zq^2;q^2)_{k_{r-1}}\prod_{j=1}^{r-2}(q;q)_{k_j-k_{j+1}}}\\
=\frac {(q;q^2)_\infty(z^rq^r,z^{-r}q^r,q^{2r};q^{2r})_\infty}
{(zq;q)_\infty(q/z;q^2)_\infty},
\end{equation}
for complex $z$ such that
$z\not\in\{q^{-2},q^{-4},q^{-6},\ldots\}$.
\end{Theorem}
\begin{proof}
  In Lemma~\ref{lem2},
  perform the substitution $(b,c,z)\mapsto(z^{\frac 12},-z^{\frac 12},z/a)$,
let $a\to\infty$ and apply \myref{jtpi}.
\end{proof}
For $z=1$ \myref{rr2multlat} reduces to the $i=r$ case of \myref{rr2mult}.
For $r=2$ \myref{rr2multlat} reduces to a special case of
Ramanujan's ${}_1\psi_1$ summation \myref{1psi1}.

Next we give a multilateral generalization of Theorem~\ref{thm:bilcc}.
(Again, for reasons of absolute convergence, we are only able to apply
Lemma~\ref{lem2}, i.e.\ the $i=r$ case of Lemma~\ref{lem}.
The $c=1$ cases of the latter could be applied to obtain multisum
identities which would be naturally bounded from below,
such as the original Andrews--Gordon identities.
However, in this work we are after \textit{multilateral} identities.)
\begin{Theorem}\label{thm:multcc}
We have for $r\ge 2$ the following multilateral summations:
\begin{subequations}\label{multcc}
\begin{align}
\label{multccb}
\sum_{\kL}&
\frac{(-q^{1-k_1}/z;q)_{k_1}\,q^{\sum_{j=1}^{r-1} k_j(k_j-1)}}
{(z^2q;q^2)_{k_{r-1}}\prod_{j=1}^{r-2}(q;q)_{k_j-k_{j+1}}}
z^{2\sum_{j=1}^{r-1} k_j}\notag\\
&=\frac{(-z;q)_\infty(q;q^2)_\infty
(z^{2r-1},z^{1-2r}q^{2r-1},q^{2r-1};q^{2r-1})_\infty}
{(z^2;q)_\infty(q^2/z^2;q^2)_\infty},
\end{align}
\begin{align}
\label{multcca}
\sum_{\kL}&
\frac{(-q^{2-2k_1}/z;q^2)_{k_1}\,q^{2\sum_{j=1}^{r-1} k_j^2}}
{(zq;q)_{2k_{r-1}}\prod_{j=1}^{r-2}(q^2;q^2)_{k_j-k_{j+1}}}
z^{2\sum_{j=1}^{r-1} k_j}\notag\\
&=\frac{(q/z;q)_\infty(-zq^2;q^2)_\infty
(-z^{2r-1}q^{2r-1},-z^{1-2r}q^{2r-1},q^{4r-2};q^{4r-2})_\infty}
{(z^2q^2,q^2/z^2,q;q^2)_\infty},
\end{align}
\begin{align}\notag
\label{multcca2}
\sum_{\kL}&
\frac{(-q^{2-2k_1}/z;q^2)_{k_1}\,q^{2\sum_{j=1}^{r-1} k_j(k_j-1)}}
{(z;q)_{2k_{r-1}}\prod_{j=1}^{r-2}(q^2;q^2)_{k_j-k_{j+1}}}
z^{2\sum_{j=1}^{r-1} k_j}=\frac{(q/z;q)_\infty(-z;q^2)_\infty}
{(z^2,q^2/z^2,q;q^2)_\infty}(q^{4r-2};q^{4r-2})_\infty\notag\\*
\times&\left[(-z^{2r-1}q,-z^{1-2r}q^{4r-3};q^{4r-2})_\infty+
z^{2r-2}(-z^{2r-1}q^{4r-3},-z^{1-2r}q;q^{4r-2})_\infty
\right],
\end{align}
\begin{align}
\notag
\label{multjac}
\sum_{\kL}&
\frac{q^{2\sum_{j=1}^{r-1} k_j^2}z^{2\sum_{j=1}^{r-1} k_j}}
{(z;q)_{1+2k_{r-1}}\prod_{j=1}^{r-2}(q^2;q^2)_{k_j-k_{j+1}}}
=\frac{(q/z;q)_\infty(q^{4r};q^{4r})_\infty}
{(z^2,q^2/z^2,q;q^2)_\infty}\notag\\*
&\times\left[(-z^{2r}q^{2r-1},-z^{-2r}q^{2r+1};q^{4r})_\infty+
z(-z^{2r}q^{2r+1},-z^{-2r}q^{2r-1};q^{4r})_\infty\right],
\end{align}
\end{subequations}
for complex $z$ such that the series on the left-hand sides
have no poles.
\end{Theorem}
\begin{proof}
To prove the respective identities, apply Lemma~\ref{lem2},
perform a specific substitution of variables (as specified below),
occasionally combined with taking a limit, and then apply one or two
instances of Jacobi's triple product identity \myref{jtpi}.
For \myref{multccb}, take
$(a,b,c,z)\mapsto (-z,zq^{-\frac 12},-zq^{-\frac 12},-z)$.
For \myref{multcca}, take
$(a,b,c,z,q)\mapsto (-z,z,-zq^{-1},-z,q^2)$.
For \myref{multcca2}, take
$(a,b,c,z,q)\mapsto (-z,zq^{-1},-zq^{-2},-zq^{-2},q^2)$.
Finally, for \myref{multjac}, take
$(b,c,z,q)\mapsto (z,zq^{-1},z^2/a,q^2)$, divide both sides by $(1-z)$
and subsequently let $a\to\infty$.
\end{proof}
All identities from Theorem~\ref{thm:multcc}
reduce to multisum generalizations of corresponding unilateral identities
discussed after Theorem~\ref{thm:bilcc}.
For instance, we have
\begin{equation}
\label{multjack12}
\sum_{\kL}
\frac{q^{2\sum_{j=1}^{r-1} k_j(k_j+\delta)}}
{(q;q)_{\delta+2k_{r-1}}\prod_{j=1}^{r-2}(q^2;q^2)_{k_j-k_{j+1}}}
=\frac{(-q^{2r(1+\delta)-1},-q^{2r(1-\delta)+1},q^{4r};q^{4r})_\infty}
{(q^2;q^2)_\infty},
\end{equation}
where $\delta =0,1$, which generalizes \myref{jack1} and \myref{jack2},
respectively.
The $\delta=0$ case is obtained by multiplying both sides of \myref{multjac}
by $(1-z)$ and letting $z\to 1$, while the $\delta=1$ case is obtained from
\myref{multjac} by letting $z\to q$.
We leave other specializations of identities from Theorem~\ref{thm:multcc}
which generalize classical unilateral summations to the reader.

When we replace $q$ by $q^{2r-1}$ in \myref{multcca2} and set $z=-q^{-1}$
or $z=-q$, we obtain the following two multilateral summations
generalizing Corollary~\ref{cor:bilcs}.
\begin{Corollary}
For $r\ge 2$ we have the following multilateral summations:
\begin{subequations}
\begin{align}
\sum_{\kL}&
\frac{(q^{1-2(2r-1)k_1};q^{2(2r-1)})_{k_1}\,q^{2(2r-1)\sum_{j=1}^{r-1} k_j^2
+4(r-1)\sum_{j=1}^{r-1}k_j}}
{(-q^{2r-2};q^{2r-1})_{1+2k_{r-1}}
\prod_{j=1}^{r-2}(q^{2(2r-1)};q^{2(2r-1)})_{k_j-k_{j+1}}}\notag\\
&=\frac{(q^{4r-3},q^{2(2r-1)})_\infty
(q^{2(2r-1)},q^{4(r-1)(2r-1)},q^{2(2r-1)^2};q^{2(2r-1)^2})_\infty}
{(q;q^{2r-1})_\infty(q^{2r-1},q^{4(r-1)};q^{2(2r-1)})_\infty},\\
\sum_{\kL}&
\frac{(q^{4r-3-2(2r-1)k_1};q^{2(2r-1)})_{k_1}\,q^{2(2r-1)\sum_{j=1}^{r-1} k_j^2
-4(r-1)\sum_{j=1}^{r-1}k_j}}
{(-q;q^{2r-1})_{2k_{r-1}}
\prod_{j=1}^{r-2}(q^{2(2r-1)};q^{2(2r-1)})_{k_j-k_{j+1}}}\notag\\
&=\frac{(q,q^{2(2r-1)})_\infty
(q^{2(2r-1)},q^{4(r-1)(2r-1)},q^{2(2r-1)^2};q^{2(2r-1)^2})_\infty}
{(q^{2(r-1)};q^{2r-1})_\infty(q^2,q^{2r-1};q^{2(2r-1)})_\infty}.
\end{align}
\end{subequations}
\end{Corollary}

Finally, we have the following generalization of Corollary~\ref{cor:biljacs}:
\begin{Corollary}\label{cor:multjacs}
For $r\ge 2$ we have the following multilateral summations:
\begin{subequations}
\begin{align}
\label{multjacsa}
\sum_{\kL}
\frac{q^{4r\sum_{j=1}^{r-1} k_j^2-2(2r-1)\sum_{j=1}^{r-1}k_j}}
{(q;q^{4r})_{k_{r-1}}(-q^{2r+1};q^{4r})_{k_{r-1}}
\prod_{j=1}^{r-2}(q^{4r};q^{4r})_{k_j-k_{j+1}}}\notag&\\*
=\frac{(q^{2r},q^{4r-1};q^{4r})_\infty
(q^{4r},q^{4r(2r-1)},q^{8r^2};q^{8r^2})_\infty}
{(q^2,q^{2r-1};q^{4r})_\infty(q^{4r},q^{2(4r-1)};q^{8r})_\infty}&,\\
\notag
\label{multjacsb}
\sum_{\kL}
\frac{q^{4r\sum_{j=1}^{r-1} k_j^2+2(2r-1)\sum_{j=1}^{r-1}k_j}}
{(q^{2r-1};q^{4r})_{1+k_{r-1}}(-q^{4r-1};q^{4r})_{k_{r-1}}
\prod_{j=1}^{r-2}(q^{4r};q^{4r})_{k_j-k_{j+1}}}\notag&\\*
=\frac{(q^{2r},q^{2r+1};q^{4r})_\infty
(q^{4r},q^{4r(2r-1)},q^{8r^2};q^{8r^2})_\infty}
{(q,q^{2(2r-1)};q^{4r})_\infty(q^{4r},q^{2(2r+1)};q^{8r})_\infty}&,\\
\notag
\label{multjacsc}
\sum_{\kL}
\frac{q^{4r\sum_{j=1}^{r-1} k_j^2-2(2r-1)\sum_{j=1}^{r-1}k_j}}
{(q^{2r+1};q^{4r})_{k_{r-1}}(-q;q^{4r})_{k_{r-1}}
\prod_{j=1}^{r-2}(q^{4r};q^{4r})_{k_j-k_{j+1}}}\notag&\\*
=\frac{(q^{2r-1},q^{2r};q^{4r})_\infty
(q^{4r},q^{4r(2r-1)},q^{8r^2};q^{8r^2})_\infty}
{(q^2,q^{4r-1};q^{4r})_\infty(q^{2(2r-1)},q^{4r};q^{8r})_\infty}&,\\
\notag
\label{multjacsd}
\sum_{\kL}
\frac{q^{4r\sum_{j=1}^{r-1} k_j^2+2(2r-1)\sum_{j=1}^{r-1}k_j}}
{(q^{4r-1};q^{4r})_{k_{r-1}}(-q^{2r-1};q^{4r})_{1+k_{r-1}}
\prod_{j=1}^{r-2}(q^{4r};q^{4r})_{k_j-k_{j+1}}}\notag&\\*
=\frac{(q,q^{2r};q^{4r})_\infty
(q^{4r},q^{4r(2r-1)},q^{8r^2};q^{8r^2})_\infty}
{(q^{2r+1},q^{2(2r-1)};q^{4r})_\infty(q^2,q^{4r};q^{8r})_\infty}&.
\end{align}
\end{subequations}
\end{Corollary}
To deduce the multilateral identities in Corollary~\ref{cor:multjacs},
first replace $q$ by $-q^{2r}$ in \myref{multjac} and then
put $z=q^{2r-1}$ to deduce \myref{multjacsb}
or $z=q^{1-2r}$ to deduce \myref{multjacsc}.
The identities in \myref{multjacsa} and \myref{multjacsd} follow by
replacing $q$ by $-q$ in \myref{multjacsc} and \myref{multjacsb}, respectively.

\section{Concluding remarks}\label{sec:conc}

In this paper, we derived a number of bilateral and multilateral identities
of the Rogers--Ramanujan type by analytic means.
The closed form bilateral summations exhibited here appear
to be the very first of their kind.
We expect that more identities of this kind can be found.
(In fact, after an earlier version of this paper appeared as a preprint,
some results were established in \cite{WYH} which are closely related to
Theorem~\ref{thm5} and the results described in Remark~\ref{remga}.)
Their very compact form and beauty suggests that these objects
merit further study.

In view of the well-established connections of the
classical Rogers--Ramanujan identities to various areas in mathematics
and in physics (including combinatorics, number theory,
orthogonal polynomials, probability theory,
statistical mechanics, representations of Lie algebras,
vertex operator algebras, knot theory and conformal field theory)
and in view of the connections to modularity which in Remark~\ref{rem:th}
were made explicit (at least for the series appearing in
Corollary~\ref{cor:bilrr}), we speculate
that similar connections to other areas can be established
for some of the bilateral identities presented in this paper.

The deeper reasons for the existence of the here presented bilateral
and multilateral identities are currently unclear. One might ask if
there is a yet to be
explained link with representation theory and wonder whether similar
results hold for Rogers--Ramanujan identities for more complicated
representations, such as for those appearing in \cite{ASW,CDU,KR,TT,W1}.
A principal candidate to obtain multilateral extensions of those identities
would be the constructive basic hypergeometric method
which was successfully applied in this paper. Unfortunately,
we have no idea which basic hypergeometric transformation
one would have to start with to successfully apply this method
in those more complicated settings.
Another natural candidate to derive further multilateral extensions
of Rogers--Ramanujan type identities would be the
functional equation method employed in in Remark~\ref{rem:ref}.
Unfortunately, the author was unable to find such identities,
that would extend those in \cite{ASW,CDU,KR,TT,W1},
by this method (or by other means, so far).
It is feasible that the product sides of the sought multilateral
identities would contain more than two terms
(whereas the product sides of the bilateral and multilateral
identities in this paper all contain at most two
terms) and would thus be rather difficult to find.

We note that we are able to give pure combinatorial interpretations of
the left- and right-hand sides of several bilateral identities
(such as those appearing in Corollaries~\ref{cor:bilrr} and
\ref{cor:bilgg}) in the spirit of the well-known partition-theoretic
interpretations of the classical Rogers--Ramanujan identities
\myref{rr} provided by MacMahon~\cite{Mac} and Schur~\cite{Sc}.
However, having such interpretations of the left-
and right-hand sides of the respective identities alone does not
automatically bring forward the bijective proofs one would like to have.
Already in the classical case the problem of finding a direct,
completely bijective, proof of \myref{rra} or of \myref{rrb}
has shown to be rather intractable.
Since the formulations of our combinatorial interpretations
of the new bilateral identities
are rather lengthy (and they do not help us to prove the identities),
we defer the details to another paper focused on those combinatorial
interpretations.

On the conceptual level the question arises whether the work in this paper
tells us anything new about the classical case.
On one hand it is interesting to observe that one
bilateral identity may contain different unilateral identities of interest.
Examples are the bilateral identities in \myref{bilcca} and
\myref{bilcca2} which each include three different unilateral identities,
as made explicit in the discussion after Theorem~\ref{thm:bilcc}.
On the other hand we would like to emphasize that the derivations
of our bilateral identities of the Rogers--Ramanujan type
do not require the combination of two unilateral sums into a bilateral sum
(such as by replacing the summation index $k$ in the second sum by $-1-k$),
which one usually requires, before applying Jacobi's triple product identity
in order to obtain the respective summations.
In this respect, our derivations are very natural and straightforward
while furnishing more general results than in the
classical unilateral cases.

\medskip
\textbf{Acknowledgments.}\ 
I would like to thank Ole Warnaar for stimulating discussions that
helped to improve the paper.
Further, I would like to thank George Andrews for making me aware of
his results in \cite{An70}. I am also indebted to Tim Huber for
sharing his insight related to Remark~\ref{rem:th}. Finally,
I would like to thank Christian Krattenthaler and Wadim Zudilin
for their interest and comments.
This research was partly supported by FWF Austrian Science Fund
grants F50-08 and P32305.

\end{document}